\documentclass[12pt]{amsart}

\usepackage{mathrsfs, amssymb,amsthm, amsfonts, amsmath}
\usepackage[all]{xy}
\usepackage{upgreek}
\usepackage{amsmath}

\usepackage{tabularx,multirow,makecell}
\newcolumntype{Y}{>{\centering\arraybackslash}X}
\usepackage{setspace}
\usepackage{textcomp}
\usepackage{hyperref}

\theoremstyle{plain}
\newtheorem{theorem}{Theorem}
\newtheorem{lemma}{Lemma}

\theoremstyle{definition}
\newtheorem*{definition*}{Definition}
\newtheorem{remark}{Remark}

\newcommand{\QQ}{\mathbb{Q}}

\newcommand{\ZZ}{\mathbb{Z}}

\newcommand{\Aut}{\operatorname{Aut}}
\newcommand{\GL}{\operatorname{GL}}

\newcommand{\Bir}{\operatorname{Bir}}

\def \le {\leqslant}

\date{}

\title{Automorphism groups of Moishezon threefolds}

\author{Yu.~G.~Prokhorov, \quad C.~A.~Shramov}
\email{prokhoro@mi-ras.ru, \quad costya.shramov@gmail.com}
\address{Steklov Mathematical Institute of Russian Academy of Sciences, 8 Gubkina st., Moscow, 119991, Russia
\newline
National Research University Higher School of Economics, Laboratory of Algebraic Geometry, 6 Usacheva str., Moscow, 119048, Russia}

\thanks{This work is supported by the Russian Science Foundation under grant \textnumero 18-11-00121.}

\begin{document}

\begin{abstract}
We study automorphism groups of
Moishezon threefolds and show that such groups are always Jordan.
\end{abstract}

\maketitle

\section{Introduction}

Jordan property plays an important role in the study of automorphism groups
of algebraic varieties and complex manifolds.
Following~\cite[Definition~2.1]{Popov2011}, we say that a group~$\Gamma$ is
\emph{Jordan}
(or has \emph{Jordan property})
if there is a constant $J=J(\Gamma)$
such that every finite subgroup~\mbox{$G\subset\Gamma$}
contains a normal abelian subgroup $A\subset G$ of index at most~$J$.

The groups that enjoy Jordan property include:
general linear groups $\GL_n(\Bbbk)$, where~$\Bbbk$ is a field of zero characteristic
(see for instance~\mbox{\cite[Theorem~36.13]{Curtis-Reiner-1962}});
groups of birational selfmaps of rationally connected
algebraic varieties (see \cite[Theorem~1.8]{ProkhorovShramov-RC} and~\mbox{\cite[Theorem~1.1]{Birkar}});
groups of birational selfmaps of non-uniruled algebraic varieties
(see \cite[Theorem~1.8]{Prokhorov-Shramov-2013});
many diffeomorphism groups of smooth compact real manifolds
(see for instance \cite{Riera-Spheres}, \cite{Riera-OddCohomology}).
One of the most beautiful results concerning Jordan property for
groups of geometric origin
is the following theorem due to Sh.~Meng and D.-Q.~Zhang.

\begin{theorem}[\cite{MengZhang}]
\label{theorem:Meng-Zhang}
Let $X$ be a projective variety over a field of
zero characteristic. Then the group $\Aut(X)$ is Jordan.
\end{theorem}

The paper \cite{Kim} provides a generalization of Theorem~\ref{theorem:Meng-Zhang}
to the case of compact K\"ahler manifolds.

Note that according to  \cite{Zarhin10} there exist projective surfaces with non-Jordan
groups of birational selfmaps;
furthermore, there exist smooth compact four-dimensional real manifolds with non-Jordan
diffeomorphism groups, see \cite{CsikosPyberSzabo}.
There is a complete classification of two- and three-dimensional
projective varieties with non-Jordan groups of birational selfmaps over
algebraically closed
fields of zero characteristic, see \cite[Theorem~2.32]{Popov2011} and~\mbox{\cite[Theorem~1.8]{Prokhorov-Shramov-3folds}}.
Moreover, it is known that the birational automorphism group of any
non-projective compact complex surface is Jordan, see \cite{ProkhorovShramov-CCS},
but in higher dimensions there are no significant results in this direction yet.
With this in mind, it looks interesting to study automorphism groups of
various classes of compact complex manifolds from the point of view of Jordan property.

Recall that a compact complex space
$X$ is said to be \emph{Moishezon} if the transcendence degree of its field of meromorphic
functions is maximal, that is, equals the dimension of $X$.
Every proper algebraic variety is a Moishezon compact complex space.
Every Moishezon compact complex space is birational to
a projective variety
(see~\mbox{\cite[Theorem~1]{Moishezon-1966}} or~\mbox{\cite[Theorem~$3_N$]{Moishezon-1966a}}).
While every compact complex curve is projective, there exist
two-dimensional Moishezon compact complex spaces that are not projective
(see for instance~\mbox{\cite[Example VII.6.26]{several-complex-variables-7}}).
A smooth Moishezon compact complex space is called a
\emph{Moishezon manifold}.
For every Moishezon compact complex space, one can construct a resolution of singularities
that is a Moishezon manifold (and even a smooth projective variety), see~\mbox{\cite[Theorem~$3_N$]{Moishezon-1966a}}.
Every two-dimensional Moishezon manifold is projective
(see~\mbox{\cite[Corollary~IV.6.5]{BHPV-2004}}).
There are well-known examples of three-dimensional non-projective
Moishezon manifolds, see for instance~\mbox{\cite[\S3]{Moishezon-1966b}}
and~\mbox{\cite[Examples~VII.6.20--VII.6.21]{several-complex-variables-7}}.

The purpose of this paper is to prove the following result that is to some extent
analogous to Theorem~\ref{theorem:Meng-Zhang}.

\begin{theorem}\label{theorem:main}
Let $X$ be a three-dimensional Moishezon compact complex space.
Then the group~\mbox{$\Aut(X)$} is Jordan.
\end{theorem}

It would be interesting to find out if there is a generalization of
Theorem~\ref{theorem:main} to the case of Moishezon compact complex spaces of arbitrary dimension.

\smallskip
We are grateful to Andreas H\"oring who spotted a gap in the first version
of our arguments.

\section{Some projectivity criteria}

In this section we collect several assertions on projectivity
of certain Moishezon varieties.

\begin{definition*}
A divisor $A$ is  \textit{strongly numerically effective}, if $A\cdot C>0$ for every curve~$C$.
\end{definition*}

\begin{lemma}\label{lemma:projectivity}
Let $X$ be a Moishezon threefold, and let $A$ be a strongly numerically effective divisor
on $X$. Suppose that for some
$n>0$ the linear system $|nA|$ has no fixed components.
Then $A$ is ample. In particular, the manifold $X$ is projective.
\end{lemma}

\begin{proof}
Follows from Nakai--Moishezon ampleness criterion for
Moishezon compact complex spaces, see~\mbox{\cite[Theorem~6]{Moishezon-1966}}.
Indeed, by assumption we have~\mbox{$A\cdot C>0$} for every curve~$C$.
Furthermore, since the divisor $A$ is numerically effective and big,
one has~\mbox{$A^3>0$}.

Let $S\subset X$ be an irreducible surface (that is, a two-dimensional compact complex
subspace). Then $S$ is a Moishezon compact complex space by~\mbox{\cite[Theorem~3]{Moishezon-1966}}.
In particular, $S$ contains curves.
The restriction $A_S=A|_S$ is an effective divisor on~$S$,
and one has~\mbox{$A_S\neq 0$}, because $A_S$ has positive intersections with curves
on $S$.
Therefore, we see that~\mbox{$A^2\cdot S>0$}. Hence $A$ is ample by Nakai--Moishezon criterion.
\end{proof}

\begin{lemma}\label{lemma:surfaces}
Let $g\colon \hat{F}\to F$ be a surjective morphism of smooth
compact complex surfaces.
Let $\hat{A}$ be an ample divisor on $\hat{F}$.
Then the divisor $A=g_*\hat{A}$ is ample.
\end{lemma}

\begin{proof}
By adjunction formula, we have $A\cdot C>0$ for every curve $C$ on $F$, and also~\mbox{$A^2>0$}.
Hence $A$ is ample by Nakai--Moishezon criterion.
\end{proof}

\begin{remark}
The assertion of Lemma~\textup{\ref{lemma:surfaces}} fails in the case when the surface
$F$ is singular. However, it still holds if $F$ is a
two-dimensional normal compact complex space
with $\QQ$-factorial singularities.
\end{remark}

For a contraction $h\colon Y \to Z$ by $\Aut(Y; h)$ we denote the subgroup of $\Aut(Y)$
that consists of all automorphisms such that $h$ is equivariant with respect to their action.
By $\Aut(Y; h)'$  we denote the maximal
subgroup of $\Aut(Y; h)$ that acts trivially on~\mbox{$H^*(Y,\ZZ)$}.
Note that~\mbox{$\Aut(Y; h)'$} is a normal subgroup of $\Aut(Y; h)$,
and the quotient~\mbox{$\Aut(Y; h)/\Aut(Y; h)'$}
has bounded finite subgroups by the classical theorem of Minkowski. Therefore, the group $\Aut(Y; h)$ has bounded
finite subgroups if and only if this holds for $\Aut(Y; h)'$.

\begin{lemma}\label{lemma:3-1}
Let $Y$ be a Moishezon threefold.
Let~\mbox{$h\colon Y \to Z$} be a contraction to a non-rational curve.
Furthermore, assume that the image of the homomorphism~\mbox{$\Aut(Y; h)'\to \Aut(Z)$} is infinite.
Then the threefold $Y$ is projective.
\end{lemma}
\begin{proof}
Let $g\colon \hat Y\to Y$ be a birational morphism
such that $\hat Y$ is projective; such a morphism
always exists, see~\mbox{\cite[Theorem~$3_N$]{Moishezon-1966a}}.
Let $\hat A$ be an ample divisor on $\hat Y$.
Set $A=g_*\hat A$. Then the divisor $A$ is big.
By projection formula,
there is at most a \textit{finite} number of
curves~\mbox{$C_i\subset Y$} such that~\mbox{$A\cdot C_i\le 0$}
(these curves must be contained in the image of the $g$-exceptional divisor).

Note that for any $\delta\in \Aut(Y; h)'$ we have
$$
A\cdot \delta{(C_i)}=A\cdot C_i\le 0.
$$
Hence the set $\{C_i\}$ of all such curves is invariant under $\Aut(Y; h)'$.
By our assumptions $Z$ is an elliptic curve and the image of $\Aut(Y; h)'\to \Aut(Z)$ consists of translations. This implies that no curve $C_i$ can be contained in a fiber of~$h$, i.e. all the curves $C_i$ dominate $Z$.
Indeed, otherwise the group $\Aut(Y; h)'$ would preserve a non-empty finite subset of $Z$ that
consists of the images of the curves $C_i$ contained in the fibers of $h$, and thus~\mbox{$\Aut(Y; h)'$}
would be finite.

Therefore, for a sufficiently ample divisor $D$ on~$Z$ one has~\mbox{$(A+h^*D)\cdot C_i>0$} for all $C_i$.
Thus $A'=A+h^*D$ is a big and strongly numerically effective divisor.
Hence it is ample by Lemma~\ref{lemma:projectivity},
and the threefold~$Y$ is projective.
\end{proof}

\begin{lemma}\label{lemma:3-2}
Let $Y$ be a Moishezon threefold, and let~\mbox{$h\colon Y \to Z$}
be a contraction to a  non-ruled surface.
Assume that the image $\Upsilon$ of the homomorphism~\mbox{$\Aut(Y; h)'\to \Aut(Z)$} has unbounded finite subgroups.
Then
the threefold $Y$ is projective.
\end{lemma}

\begin{proof}
Similarly to the proof of Lemma~\ref{lemma:3-1}, consider a birational morphism~\mbox{$g\colon \hat Y\to Y$}
such that $\hat Y$ is projective. Let $\hat A$ be an ample divisor on~$\hat Y$, and set~\mbox{$A=g_*\hat A$}.
There is at most a finite number of curves $C_i\subset Y$ such that~\mbox{$A\cdot C_i\le 0$}.
As in the proof of Lemma~\ref{lemma:3-1} the set $\{C_i\}$ of all such curves is invariant under $\Aut(Y; h)'$. Suppose that some curve $C_i$ is contained in a fiber of~$h$.
Let $z=h(C_i)$, let $\Aut(Z,z)\subset \Aut(Z)$ be the stabilizer of $z$, and let $\Upsilon_z=\Upsilon\cap  \Aut(Z,z)$. The index $[\Upsilon: \Upsilon_z]$ is finite.
Therefore, $\Upsilon_z$ has unbounded finite subgroups. By \cite[Lemma~2.5]{MengZhang}
the finite subgroups of~\mbox{$\Upsilon_z/(\Upsilon_z\cap \Aut_0(Z))$}
are bounded, where $\Aut_0(Z)$ is the connected component of identity in $\Aut(Z)$. On the other hand,
since $Z$ is not ruled,  $\Aut_0(Z)$
is either trivial or is an abelian variety and so the stabilizer
$\Aut_0(Z, z)$ of $z$ in $\Aut_0(Z)$ must be trivial.
The contradiction shows that  no curve $C_i$ can be contained in a fiber of $h$.
Thus for a sufficiently ample divisor $D$ on~$Z$ the divisor~\mbox{$A'=A+h^*D$} is
big and strongly numerically effective.
Hence $A'$ is ample by Lemma~\ref{lemma:projectivity}, and the threefold~$Y$ is projective.
\end{proof}

Note that the construction of~\mbox{\cite[Example~VII.6.20]{several-complex-variables-7}}
allows one to obtain an example of a non-projective
Moishezon threefold such that the base of its maximal rationally connected fibration has arbitrary dimension
(from~$0$ to~$3$).

\section{Proof of Theorem~\ref{theorem:main}}

In this section we prove Theorem~\ref{theorem:main}.

Suppose that the group $\Aut(X)$ is not Jordan.
Since~\mbox{$\Aut(X)$} is a subgroup of the group~\mbox{$\Bir(X)$} of birational selfmaps of $X$,
we conclude that $\Bir(X)$ is not Jordan either.
Note that $\Bir(X)$ is isomorphic to the group of birational selfmaps of
some projective variety~$\hat{X}$ birational to the compact complex space~$X$.
According to~\mbox{\cite[Theorem~1.8(ii)]{Prokhorov-Shramov-2013}}, the variety~$\hat{X}$ is uniruled,
and by~\mbox{\cite[Theorem~1.8]{ProkhorovShramov-RC}} and~\mbox{\cite[Theorem~1.1]{Birkar}} it is not rationally
connected. Since~$X$ is birational to~$\hat{X}$, we see that~$X$ is also uniruled but not
rationally connected. There exists the maximal rationally connected fibration~\mbox{$f\colon X \dashrightarrow V$},
and one has~\mbox{$0<\dim V<\dim X$}, see~\mbox{\cite[Theorem~5.5.4]{Kollar-1996-RC}}.
The compact complex space $V$ is Moishezon by~\mbox{\cite[Theorem~2]{Moishezon-1966}}.
Note that the maximal rationally connected fibration is defined
only as a rational
map. Thus, resolving the singularities of $V$, we may assume that $V$ is smooth.

One of our main tools is the following result that is implied by the existence of
a canonical resolution of singularities
(see~\cite[\S13]{Bierstone-Milman-1997}).

\begin{theorem}\label{theorem:BierstoneMilman}
Let $M$ be an irreducible compact
complex manifold, and let $W\subset M$ be its compact complex subspace.
Then there exists a sequence of blow ups~\mbox{$\pi\colon \tilde{M}\to M$} with smooth centers
such that the union of the proper transform
$\pi^{-1}W$ with the exceptional locus
$E$ of $\pi$ is a simple normal crossing divisor. Moreover, the morphism
$\pi$ is canonical in the following sense:
every automorphism $M \to M$ preserving $W$ can be extended to an
automorphism~\mbox{$\tilde{M}\to\tilde{M}$} that commutes with~$\pi$.
\end{theorem}

Theorem~\ref{theorem:BierstoneMilman} allows us to prove the following result.

\begin{lemma}\label{lemma:diagram}
Let $X$ be a Moishezon compact complex space. Let~\mbox{$f\colon X \dashrightarrow V$} be the maximal
rationally connected fibration, where we choose $V$ to be smooth.
Suppose that~\mbox{$\dim V\le 2$}. Let $Z$ be the minimal model of~$V$.
Then there is an $\Aut(X)$-equivariant commutative diagram
\begin{equation}\label{diagram}
\vcenter{
\xymatrix@R=17pt{
X\ar@{-->}[dr] & Y\ar [d]^{h}\ar[l]
\\
& Z
}}
\end{equation}
Here $Y$ is a Moishezon manifold, $Z$ is smooth and projective,
$Y \to X$ is a birational morphism, and~\mbox{$h\colon Y \to Z$} is the maximal rationally connected
fibration for~$Y$.
\end{lemma}

\begin{proof}
Since the manifold $V$ is Moishezon and has dimension at most $2$,
it is projective. Hence its minimal model $Z$ is projective (and smooth) as well.

Recall that the group $\Aut(X)$ acts on $V$ by birational maps (possibly non-faithfully).
Since $Z$ is a minimal model of $V$, the group~\mbox{$\Aut(V)$} acts on $Z$ biregularly.
The composition~\mbox{$\sigma\colon X\dasharrow V$} of $f$ with the contraction $V\to Z$
is an $\Aut(X)$-equivariant map. Consider the closure
$\bar{\Gamma}_\sigma$ of the graph of this map in $X\times Z$.
Since the action of $\Aut(X)$ on~\mbox{$X\times Z$} is biregular, the action of $\Aut(X)$ on
$\bar{\Gamma}_\sigma$ is biregular as well. Finally, let~$Y$ be
the canonical resolution of singularities of $\bar{\Gamma}_\sigma$
provided by Theorem~\ref{theorem:BierstoneMilman}. The action of $\Aut(X)$ on $Y$
is again biregular, which gives us the commutative $\Aut(X)$-equivariant diagram~\eqref{diagram}.
\end{proof}

Apply Lemma~\ref{lemma:diagram} to our Moishezon compact complex space
$X$.
We have an embedding~\mbox{$\Aut(X)\subset\Aut(Y)$}.

Since the map $h\colon Y\to Z$ is the maximal rationally connected fibration for $Y$,
we see that the group $\Aut(Y)$ acts (possibly non-faithfully) on $Z$, and the map
$h$ is $\Aut(Y)$-equivariant.
Let $\Aut(Y)_h$ be
the subgroup of $\Aut(Y)$
that consists of all automorphisms whose action is fiberwise with respect to~$h$,
and let $\Upsilon\cong\Aut(Y)/\Aut(Y)_h$ be the image of the group $\Aut(Y)$ in $\Aut(Z)$. All finite subgroups
of $\Aut(Y)_h$ act faithfully on a non-multiple fiber
of $h$, and thus $\Aut(Y)_h$ is Jordan by \cite[Theorem~1.5]{ProkhorovShramov-CCS}.
This implies that if $\Upsilon$ has bounded finite subgroups, then
the group $\Aut(Y)$ is
Jordan as well. Therefore, we will assume that the group $\Upsilon$ has unbounded finite subgroups.

Suppose that the dimension of $Z$ equals $1$.
Then $Z$ is a non-rational curve.
In this case~$Y$ is projective by Lemma~\ref{lemma:3-1}.

Suppose that the dimension of $Z$ equals $2$.
Then $Z$ is a non-ruled surface.
In this case~$Y$ is projective by Lemma~\ref{lemma:3-2}.

Therefore, we see that under our assumptions
the group $\Aut(X)$ is contained in the automorphism group
of the projective variety~$Y$. Now Theorem~\ref{theorem:main} follows from Theorem~\ref{theorem:Meng-Zhang}.


\end{document}